\documentclass[10pt, reqno]{amsart}
\usepackage{amssymb, amsmath, amsthm, amsfonts, enumerate, palatino, hyperref, physics}
\usepackage[normalem]{ulem}
\usepackage[parfill]{parskip}
\usepackage[utf8]{inputenc}
\usepackage{fullpage}
\usepackage{upgreek}
\usepackage{mathrsfs}
\usepackage{geometry}
\usepackage{amssymb}
\usepackage{mathtools}
\usepackage{esint}
\usepackage{nicefrac}
\usepackage{fancyhdr}
\usepackage{fixme}

\theoremstyle{plain}
\newtheorem{prop}{Proposition}
\newtheorem{lemma}[prop]{Lemma}
\newtheorem{thm}[prop]{Theorem}

\newtheorem{cor}[prop]{Corollary}

\theoremstyle{remark}
\newtheorem{rmk}[prop]{Remark}

\theoremstyle{definition}

\theoremstyle{definition}
\newtheorem{defn}[prop]{Definition}

\newcommand{\R}{\mathbb{R}}

\newcommand{\defeq}{:=}

\input{tcilatex}
\begin{document}
\author{Usman Hafeez, Théo Lavier, Lucas Williams, and Lyudmila Korobenko}
\title{Orlicz-Sobolev Inequalities and the Dirichlet Problem for Infinitely Degenerate Elliptic Operators}
\maketitle

\begin{abstract}
We investigate a connection between solvability of the Dirichlet problem for an infinitely degenerate elliptic operator and the validity of an Orlicz-Sobolev inequality in the associated subunit metric space. For subelliptic operators it is known that the classical Sobolev inequality is sufficient and almost necessary for the Dirichlet problem to be solvable with a quantitative bound on the solution \cite{Sawyer}. When the degeneracy is of infinite type, a weaker Orlicz-Sobolev inequality seems to be the right substitute \cite{Luda}. In this paper we investigate this connection further and reduce the gap between necessary and sufficient conditions for solvability of the Dirichlet problem.
\end{abstract}

\section{Introduction}
Consider the Dirichlet problem with a divergence form (degenerate) elliptic operator
\begin{equation}\label{Problem}
\begin{cases}
\nabla\cdot A \nabla u = f\  \  \text{in}\  \  \Omega \\ 
u|_{\partial\Omega}=0
\end{cases},
\end{equation}
where $A$ is nonnegative semidefinite and has bounded measurable coefficients, and $\Omega$ is a bounded domain in $\R^n$ with sufficiently smooth boundary. We are interested in establishing sharp conditions on the matrix $A$ that guarantee existence of bounded weak solutions. More precisely, we are looking for a function $u$ from the degenerate Sobolev space $\left(W_{A}^{1,2}\right)_0(\Omega)$ satisfying
\[
\int\nabla u\cdot A\nabla\varphi=-\int f\varphi
\]
for every test function $\varphi$ (in which case we say that $u$ is a weak solution of (\ref{Problem})), as well as the qualitative estimate
\[
||u||_{L^{\infty}(\Omega)}\leq C||f||_X
\]
for some appropriate normed space $X$. The case when $A$ is elliptic has been completely settled by Nash \cite{Nash}, Moser \cite{Mos}, and DeGiorgi \cite{DeG}, and is now considered a classical theory \cite{GilTrud}. When the eigenvalues of the matrix $A$ are allowed to vanish, i.e. the operator is degenerate elliptic, the theory is far from complete. There are generally two cases considered in the literature: finite vanishing with rough coefficients, and infinite vanishing with smooth coefficients. In the case of finite vanishing, the first generalizations of the Moser-DeGiorgi theory are due to Fabes, Kenig, and Serapioni \cite{FKS82} and Franchi and Lanconelli \cite{FrLan2}. The latter deals with the case when one of the eigenvalues of $A$ is constant, while others may vanish to finite order. Franchi and Lanconelli's big idea was to use the subunit metric space associated to the operator, and adapt the classical Moser iteration to that setting. Using this approach, Sawyer and Wheeden \cite{Sawyer} built on the work of Franchi and Lanconelli, among others, to further investigate regularity questions for subelliptic operators with rough coefficients. In particular, they showed that the $(2\sigma,2)$ weak Sobolev inequality with $\sigma>1$ in the subunit metric space
\begin{equation}\label{Sob-classic_0}
\left(\frac{1}{|B|}\int_{B}|w|^{2\sigma}\right)^{\frac{1}{2\sigma}}\leq Cr\left(\frac{1}{|B|}\int_{B}|\nabla_A w|^{2}\right)^{\frac{1}{2}}+C\left(\frac{1}{|B|}\int_{B}|w|^{2}\right)^{\frac{1}{2}}
\end{equation}
for all $w\in W^{1,2}_{0}(B)$, is sufficient for solvability of the Dirichlet problem (\ref{Problem}) when $\Omega=B$, a subunit metric ball, with the quantitative estimate
\[
||u||_{L^{\infty}(B)}\leq C||f||_{L^q(B)}
\]
where $q>\sigma'$, and $\sigma'$ is the dual of $\sigma$. Moreover, if the above estimate holds for $q=\sigma'$ then Sobolev inequality (\ref{Sob-classic_0}) holds (almost necessity).
In this paper we investigate the same question for the case of the infinitely degenerate operator $L=\nabla\cdot A\nabla$. More precisely, we make use of an analogue of (\ref{Sob-classic_0}), considering the more general Orlicz spaces, $L^\phi$, instead of the traditional Lebesgue spaces. By a $(\phi,2)$ Orlicz-Sobolev inequality we mean the following
\begin{equation}\label{Sob-Orlicz}
||w||_{L^{\phi}(B,d\mu)}\leq C(r)\left(\int_{B}|\nabla_A w|^{2}d\mu\right)^{\frac{1}{2}}
\end{equation}
for all $w\in \left(W^{1,2}_{A}\right)_{0}(B)$ and some Young function $\phi$ (typically satisfying $\phi(t)>t^2$ for all $t>1$), see Section \ref{Preliminaries} for precise definitions. Here and in what follows we use the notation 
\[
d\mu=\frac{dx}{|B|},
\]
and all the integrals are taken with respect to this measure, unless otherwise stated.
There are a few recent results indicating that Orlicz-Sobolev inequalities of the type (\ref{Sob-Orlicz}) are the correct substitute for (\ref{Sob-classic_0}) when the operator is infinitely degenerate.
First, as has been shown in \cite{KoMaRi}, a classical weak Sobolev inequality (\ref{Sob-classic_0}) implies the doubling property of the underlying metric measure space, and hence the degeneracy must be of finite type.
On the other hand, in \cite{Luda,Luda2} an abstract regularity theory for degenerate operators has been developed under the assumption of appropriate Orlicz-Sobolev inequalities (stronger versions of (\ref{Sob-Orlicz})). Moreover, for particular classes of infinitely degenerate operators these inequalities were proved to hold in the degenerate Sobolev spaces associated to the operator.

In this paper we investigate the connection between the Dirichlet problem (\ref{Problem}) and the validity of (\ref{Sob-Orlicz}).
In particular, we prove sufficiency and almost necessity of an Orlicz-Sobolev type inequality for the existence, uniqueness, and boundedness of weak solutions to \textit{degenerate} elliptic partial differential equations with homogeneous Dirichlet boundary conditions.

Our main results are as follows
\begin{thm}\label{thm:sufficiency}
Let $L=\nabla\cdot A\nabla $ with bounded measurable non-negative semidefinite matrix $A$, and $d$ a metric on $\R^n$. Suppose also that (\ref{Sob-Orlicz}) holds for all $w\in \left(W^{1,2}_{A}\right)_{0}(B)$ and the metric ball $B=\Omega\subset \R^n$ with $\phi$ satisfying $\phi(t)\geq t^2$ for all $t\geq 0$, and  $\phi(t)\geq t^2(\ln t)^N,\ N>1$, for all $t\geq 1$. If $f\in L^{\infty}(B),$ then there exists a unique weak solution $u\in \left(W^{1,2}_{A}\right)_{0}(B)$ of (\ref{Problem}) in the ball $\Omega=B$ and it satisfies
\[
||u||_{L^{\infty}(B)}\leq C||f||_{L^{\infty}(B)}.
\]
\end{thm}

\begin{thm}\label{thm:necessity}
Let $\varphi$ be a Yong function with $\tilde{\varphi}$ being its dual, and define $\phi$ by $\phi(t)=\varphi(t^2)$ for all $t\in \R$. Suppose that for every $f\in L^{\tilde{\varphi}}(B)$ there exists a unique weak solution $u\in \left(W^{1,2}_{A}\right)_{0}(B)$ of (\ref{Problem}) in the ball $\Omega=B$ which satisfies
\[
||u||_{L^{\infty}(B,d\mu)}\leq C||f||_{L^{\tilde{\varphi}}(B,d\mu)}.
\]
Then Orlicz-Sobolev inequality (\ref{Sob-Orlicz}) holds for all $w\in \left(W^{1,2}_{A}\right)_{0}(B)$.
\end{thm}
\begin{rmk}
Note that in the above theorems we do not assume that the metric $d$ is the subunit metric associated to $A$. In practice, to prove Orlicz-Sobolev inequality (\ref{Sob-Orlicz}) one would need to work in a subunit metric space \cite{FrLan2}, or a measure space associated to the operator \cite{FKS82}.
\end{rmk}

A version of the result in Theorem \ref{thm:necessity} and a sketch of the proof appears in Sections 1 and 2 of Chapter 9 in \cite{Luda2}.
It can be seen as a generalization of the subelliptic result (Lemma 102 in \cite{Sawyer}) with $L^{\varphi}$ replacing $L^{\sigma}$ and $L^{\phi}$ replacing $L^{2\sigma}$. In the subelliptic case, the requirement on the right hand side is $f\in L^{q}$ with $q>\sigma'$. In Theorem \ref{thm:necessity} we require $f\in L^{\infty}$, a strengthening of $L^{\tilde{\varphi}}$.
Note that just like in the subelliptic case, there is a gap between necessary and sufficient conditions. We suspect that the sufficient condition in Theorem \ref{thm:sufficiency} can be sharpened, but not with our current method of proof. At this point we do not know if the gap can be closed completely.

The paper is organized as follows. After giving some background and preliminaries in Section \ref{Preliminaries}, we prove the existence and global boundedness of weak solutions, Theorem \ref{thm:sufficiency}, in Section \ref{sec:sufficiency}. Section \ref{NecExtUnq} is devoted to the proof of Theorem \ref{thm:necessity}, the necessity of Orlicz-Sobolev for solvability of the Dirichlet problem with a quantitative bound. The proof follows closely the proof of Lemma 102 in \cite{Sawyer}, and it also appears in Sections 1 and 2 of Chapter 9 in \cite{Luda2}. However, the case of Orlicz-Sobolev spaces is more delicate, so we fill in the gaps and provide all the details. Finally, Section \ref{Counterexamples} provides some counterexamples demonstrating that the requirement on the right hand side in Theorem \ref{thm:sufficiency} cannot be significantly relaxed. More precisely, we give examples of equations admitting unbounded weak solutions in the case of Laplacian, subelliptic, and infinitely degenerate elliptic operators.

\section{Preliminaries}\label{Preliminaries}
\subsection{Subunit metric spaces}
We start this section with some background material on subunit metric spaces associated to degenerate operators, all of which can be found in \cite[Chapter 7]{Luda}. As mentioned in the Introduction, we do not assume the underlying metric space is the subunit metric space, however, it will be used to construct counterexamples in Section \ref{Counterexamples}. 
\subsubsection{Degenerate Sobolev spaces}
Let $A$ be nonnegative semidefinite bounded measurable matrix, and assume that $A(x)=B\left( x\right) ^{\func{tr}%
}B\left( x\right) $ where $B\left( x\right) $ is a Lipschitz continuous $%
n\times n$ real-valued matrix defined for $x\in \Omega $. We define the $A$%
-gradient by%
\begin{equation}
\nabla _{A}=B\left( x\right) \nabla \ ,  \label{def A grad}
\end{equation}%
and the associated degenerate Sobolev space $W_{A}^{1,2}\left( \Omega
\right) $ to have norm%
\begin{equation*}
\left\Vert v\right\Vert _{W_{A}^{1,2}}\equiv \sqrt{\int_{\Omega }\left(
	\left\vert v\right\vert ^{2}+\nabla v^{\func{tr}}A\nabla v\right) }=\sqrt{%
	\int_{\Omega }\left( \left\vert v\right\vert ^{2}+\left\vert \nabla
	_{A}v\right\vert ^{2}\right) }.
\end{equation*}
The space $\left( W_{A}^{1,2}\right) _{0}\left( \Omega \right) $ is defined as the closure in 
	$W_{A}^{1,2}\left( \Omega \right) $ of the subspace of Lipschitz continuous
	functions with compact support in $\Omega $.
	
	\begin{defn}\label{innerProduct}
Given $u,v\in W_A^{1,2}$, define the \emph{inner product} on the gradients of $u$ and $v$ to be
\[
\langle \nabla u,\nabla v\rangle=\nabla u\cdot A\nabla v \defeq \nabla u^{tr} A\nabla v.
\]
Furthermore, define the \emph{$A$ semi-norm} of $\nabla u$ to be 
\[[\nabla u ]_A^2\defeq \langle \nabla u,\nabla u\rangle.\]
\end{defn}
\subsubsection{Subunit metrics}
We now define subunit (or control, or Carnot-Carath\'{e}odory) metric associated to the operator $L=\nabla\cdot A\nabla$, see e.g. \cite{FrLan2}.
\begin{defn}
A \emph{subunit curve} is
Lipschitz curve $\gamma:\,[0,r]\rightarrow\Omega$ such that
$$
(\gamma'(t)\xi)^2\leq\xi'A(\gamma(t))\xi,\;\; a.e.\;t\in[0,r],\;\; \xi\in\R^n
$$
\emph{Subunit metric} is defined by
$$
d(x,y)=\inf\{r>0:\;\gamma(0)=x,\;\gamma(r)=y,\;\gamma\  is\  subunit\  in\  \Omega\}
$$
and the \emph{subunit ball} centered at $x$ with radius $r$ is
$$
B(x,r)=\{y\in\Omega:d(x,y)<r\}
$$
\end{defn}
Franchi and Lanconelli \cite{FrLan2} were the first to realize that the classical Moser iteration scheme can be adapted to certain degenerate operators (with one fixed constant eigenvalue) provided the Euclidean $\R^n$ is replaced by the subunit metric space.

\subsection{Orlicz spaces}
As mentioned in the introduction, we will work with Orlicz spaces, which can be seen as generalizations of Lebesgue spaces: power functions used do define Lebesgue spaces are replaced by more general Young functions. The material below is taken from \cite{Orlicz} 
\begin{defn} \cite{Orlicz}
A function $\theta\colon\R\to[0,\infty]$ is a \emph{Young function} if 
\begin{description}
\item[1] $\theta$ is a convex, lower semicontinuous, $[0,\infty]$- valued function on $\R$.
\item[2] $\theta$ is even and $\theta(0)=0$.
\item[3] $\theta$ and its convex conjugate, $\Tilde{\theta}$, are non-trivial; i.e. it is different from the constant function $\theta(s)=0$ for $s\in \R$.
\end{description}

Note that Properties 1 and 2 imply that any Young function is non-decreasing on $[0,\infty)$ \cite{Orlicz}.
\end{defn}

\begin{defn}\label{dual} \cite{Orlicz}
Given a Young function, $\theta$, the \emph{convex conjugate} of $\theta$, denoted $\Tilde{\theta}$, is defined as
\[
\Tilde{\theta} = \sup_{s\in\R}\{st-\theta(s)\} \in [0,\infty] \text{ for } t\in\R
\]
\end{defn}

We next define the Luxembourg norm, which in turn leads to the definition of an Orlicz space.
\begin{defn}\label{def:LuxNorm}
Let $\theta$ be a Young function, and $\Omega$ be a space with a $\sigma$-field and a $\sigma$-finite positive measure $\mu$. For any measurable function on $\Omega$ we define the \emph{Luxembourg Norm} as:
\begin{equation}\label{LuxNorm}
\|f\|_{L^\theta}\defeq \inf\left\{k>0: \int_\Omega \theta(f/k) d\mu \leq 1 \right\}
\end{equation}
where $\inf(\emptyset)=+\infty$.
\end{defn}
For a Young function $\theta$, the associated Orlicz space is defined to be 
\[
L^\theta=\{f \text{ measurable }:\|f\|_{L^\theta}<\infty\}
\]
The following proposition follows directly from definition (\ref{LuxNorm}).
\begin{prop}\label{prop:scale}
Let $\theta_1$ and $\theta_2$ be two Young functions such that $\theta_1(t)\leq \theta_2(t)$ for all $t\geq 0$. Then $L^{\theta_{2}}\subseteq L^{\theta_{1}}$, in particular, for every $f\in L^{\theta_{2}}$ there holds
\[
||f||_{L^{\theta_{1}}}\leq C ||f||_{L^{\theta_{2}}}.
\]
\end{prop}

An equivalent norm on $L^\theta$ given below is based on duality and will be used in some of the proofs contained in this paper. 
\begin{defn}{\label{OrliczNorm}}
The \emph{Orlicz Norm} of a measurable function $f$ is defined as
\begin{equation*}
|f|_{L^\theta}:=\sup \left\{ \int_\Omega fg\dd{\mu} :g\in L^{\Tilde{\theta}}\text{ and } \|g\|_{L^{\Tilde{\theta}}}\leq 1 \right\}=\sup \left\{ \int_\Omega fg\dd{\mu} :g\in L^{\Tilde{\theta}}\text{ and } \int_{\Omega}\Tilde{\theta}(g)d\mu\leq 1 \right\} .
\end{equation*}
\end{defn}
More precisely, there holds
\begin{equation}\label{eq:Or_eqiuv}
\|f\|_{L^\theta}\leq |f|_{L^\theta} \leq 2\|f\|_{L^\theta}.
\end{equation}
\begin{prop}(H\"older Inequality) \cite{Orlicz}
Given $\theta$, a Young function, $f\in L^\theta$, and $g\in L^{\Tilde{\theta}}$
\begin{equation}\label{eq:Holder}
\int |fg| d\mu \leq 2\|f\|_\theta \|g\|_{\Tilde{\theta}}.
\end{equation}
In particular, $fg\in L^1$.
\end{prop}

Finally, we define a particular family of Orlicz functions first introduced in \cite{Luda} employed in the adaptation of DeGiorgi iteration in the proof of Theorem \ref{thm:sufficiency}
\begin{defn}\label{bumpfam}
The family of \emph{Orlicz bump functions $\{\Phi_N\}_{N>1}$} is given by 
\[\Phi_N(t)=\begin{cases}
            t(\ln{t})^N, & \text{ if } t\geq E=E_N=e^{2N};\\
            (\ln{E})^N t, & \text{ if } 0\leq t \leq E=E_N=e^{2N}.
            \end{cases}
\]
\end{defn}

\section{Sufficiency}\label{sec:sufficiency}
This section is devoted to the proof of Theorem \ref{thm:sufficiency}. First we show existence and uniqueness of weak solutions and then establish the quantitative boundedness estimate.
\subsection{Existence of a unique weak solution}
The proof is based on the Lax-Milgram theorem applied to the appropriate bilinear form $B[u,v]$ defined on $\left(W_{A}^{1,2}\right)_{0} \times \left(W_{A}^{1,2}\right)_{0}$.
\begin{prop}\label{prop:Lax-Milg}
        Let $\Omega \subset \mathbb{R}^n$ be a bounded subset, and $A$ a nonnegative semidefinite $n\times n$ matrix with bounded measurable coefficients. Suppose that for every  $w \in \left({W_A}^{1,2}\right)_0(\Omega)$ the following $2-2$ Sobolev inequality holds
        \begin{equation}\label{eq:2-2-Sob}
            \int_{\Omega}|w|^{2}dx\leq C(\Omega) \int_{\Omega}|\nabla_A w|^{2}dx.
        \end{equation}
        Then the bilinear form $B : \left({W_A}^{1,2}\right)_0(\Omega) \times \left({W_A}^{1,2}\right)_0(\Omega)
         \rightarrow \mathbb{R}$ defined by:
        \[
        B[u,v]:= \int_\Omega 
        \nabla u\cdot A \nabla v
        \] 
        is bounded and coercive,i.e,
        \begin{enumerate}
            \item[(1)] There exists $\alpha > 0$ such that 
            $|B[u,v]| \leq \alpha \norm{u}_{W^{1,2}_A}\norm{v}_{W^{1,2}_A}$ for all $u,v \in 
            \left({W_A}^{1,2}\right)_0(\Omega)$.
            \item[(2)] There exists $\beta > 0$ such that 
            $\beta ||u||_{W^{1,2}_A}^{2} \leq B[u,u]$ for all 
            $u \in \left({W_A}^{1,2}\right)_0(\Omega)$.
        \end{enumerate}
    \end{prop}
        \begin{proof}
        We begin by showing $B$ is bounded.
        \begin{align*}
        \abs{B[u,v]} 
        & =\abs{\int{\nabla u \cdot A \nabla v}}\\
            & \leq \left(\int\abs{\nabla u\cdot A\nabla u}  \right)^{\frac{1}{2}} \left( \int \abs{\nabla v\cdot A\nabla v} \right)^\frac{1}{2} \\
            & \leq \left( \int u^2+\int \abs{\nabla u \cdot A\nabla u }\right)^{\frac{1}{2}} \left( \int v^2 + \int\abs{\nabla v\cdot A\nabla v} \right)^{\tfrac{1}{2}}\\
            & =  \|u\|_{W^{1,2}_A}\|v\|_{W^{1,2}_A}.
        \end{align*}
        for all $u,v \in \qty(W^{1,2}_A)_0$, where the second line
        is due to Holder's inequality. We now proceed to prove 
        (2), the coercivity of the bilinear form $B$ where
        we use Sobolev inequality (\ref{eq:2-2-Sob}) 
        \begin{equation*}
            \int_{\Omega}{u^2} \leq 
            C \int_{\Omega}{| \nabla_{A}u|^{2}} = 
            C B[u,u].
        \end{equation*}
        Namely, we have
        \begin{align*}
            B[u,u]
            & = \frac{1}{2} B[u,u] + \frac{1}{2} B[u,u]\\
            &= \frac{1}{2} \int_{\Omega}{\nabla u \cdot A \nabla u}
            + \frac{1}{2}B[u,u]\\
            &= \frac{1}{2} \int_{\Omega}
            { | \nabla_{A}u |^2 } + 
            \frac{1}{2}B[u,u]\\
            &\geq \frac{1}{2C} \int_{\Omega}{ u^2 } + \frac{1}{2}B[u,u]\\ 
            & = \frac{1}{2C} \int_{\Omega}{ u^2 } + \frac{1}{2}\int_{\Omega} \nabla u\cdot A\nabla u\\
            & \geq  \min\qty{\frac{1}{2C},\frac{1}{2}}\left( \int_{\Omega} u^2+\int_{\Omega} \nabla u \cdot A\nabla u \right)\\
            & = \beta \|u\|^2_{W^{1,2}_A}
        \end{align*}
        where $\beta = \min\{ \frac{1}{2C}, \frac{1}{2} \}$
        and $C$ is as in the aforementioned Sobolev 
        Inequality.
        
        Therefore, the bilinear form, $B$, is
        bounded and coercive.  
        \end{proof}
  We are now ready to show the existence and uniqueness of a weak solution claimed as in Theorem \ref{thm:sufficiency}.  
   \begin{thm}\label{thm:existence}
Let $L=\nabla\cdot A\nabla $ with bounded measurable non-negative semidefinite matrix $A$, and $d$ a metric on $\R^n$. Suppose also that the Sobolev inequality (\ref{eq:2-2-Sob}) holds for all $w\in \left(W^{1,2}_{A}\right)_{0}(B)$ and the metric ball $B=\Omega\subset \R^n$. If $f\in L^{\infty}(B),$ then there exists a unique weak solution $u\in \left(W^{1,2}_{A}\right)_{0}(B)$ to the following Dirichlet problem
\begin{equation}\label{problem'}
\begin{cases}
\nabla\cdot A \nabla u = f\  \  \text{in}\  \  B \\ 
u|_{\partial B}=0
\end{cases}.
\end{equation}
\end{thm} 
\begin{proof}
Consider the linear functional $(f,\cdot) : \qty(W^{1,2}_A)_0(B) \rightarrow \mathbb{R}$ defined by
\[
(f,w)=-\int_{B}fw,\quad \forall w\in \qty(W^{1,2}_A)_0(B).
\]
Since $f\in L^{\infty}(B)$ and $w\in \qty(W^{1,2}_A)_0(B)\subset L^{1}(B)$ we have
\[
|(f,w)|\leq C||f||_{L^{\infty(B)}}||w||_{W^{1,2}_A(B)},
\]
which shows that this linear functional is bounded on $\qty(W^{1,2}_A)_0(B)$.
Therefore, by Proposition \ref{prop:Lax-Milg} there exists a unique 
        element, $u \in \qty(W^{1,2}_A)_0(B)$, such that $B[u,w]=( f,w )$ for all $w\in \qty(W^{1,2}_A)_0(B)$.
By definition of $B[u,w]$ this means
\[
\int\nabla u\cdot A\nabla w=-\int fw
\]
for all $w\in \qty(W^{1,2}_A)_0(B)$,
and we conclude that
        $u$ is a unique weak solution to (\ref{problem'}).
\end{proof}
    
 \begin{cor}
 Under the assumptions of Theorem \ref{thm:sufficiency} there exists a unique weak solution to (\ref{problem'}).
 \end{cor}
 \begin{proof}
     Indeed, suppose the Orlicz-Sobolev inequality (\ref{Sob-Orlicz}) holds with $\varphi(t)\geq t^{2}$ for all $t\geq 0$. Moreover, the Orlicz space defined by $\psi(t)=t^2$ coincides with $L^2$. Therefore, by Proposition \ref{prop:scale} the classical Sobolev inequality (\ref{eq:2-2-Sob}) holds and Theorem \ref{thm:existence} applies.
 \end{proof}

\subsection{Global Boundedness of Weak Solutions.} \label{Sufficiency}

We now arrive at the proof of the global boundedness estimate for weak solutions. 
The proof closely follows the argument of \cite[Chapter 4]{Luda}. However, the Orlicz-Sobolev inequality we assume is weaker than the one in \cite{Luda}, while our assumption on the right hand side $f$ is stronger. We therefore provide the details of the arguments that are necessary to verify in this new setting.
We start with a Caccioppli inequality, which is an analogue of \cite[Proposition 24 and Corollary 25]{Luda}. 

\begin{prop}\label{cacc}
Let $u$ be a weak solution to (\ref{Problem}) on $\Omega=B$ and define $u_+=\max\qty{u,0}$, then the following Caccioppoli inequality holds on a ball $B$:
\[
\int_{\qty{x\in B:u(x)>0}} |\nabla_A u_+|^2 \dd{\mu} \leq \int_{\qty{x\in B:u(x)>0}} u_+\|f\|_{L^{\infty}} \dd{\mu}
\]
where $\dd{\mu}=\frac{\dd{x}}{\abs{B}}$.
\end{prop}
\begin{proof}
Let $v=u_+$ then we have $v\in\qty(W^{1,2}_A)_0(B)$ and therefore,
\begin{align*}
    \int_{B} \nabla u \cdot A \nabla v \dd{\mu} &= - \int_{B} f v \dd{\mu} \\
    \int_{\qty{x\in B:u(x)>0}} \nabla u \cdot A \nabla u_+ \dd{\mu} &= - \int_{\qty{x\in B:u(x)>0}} fu_+ \dd{\mu} \\
    \int_{\qty{x\in B:u(x)>0}} |\nabla_A u_+|^2 \dd{\mu} &\leq \int_{\qty{x\in B:u(x)>0}} u_+\|f\|_{L^{\infty}} \dd{\mu}.
\end{align*}
\end{proof}
\begin{cor}
Let $u$ be a weak solution to (\ref{Problem}) in $B$, and suppose that for some $P>0$ and a non-negative function $v\in W_A^{1,2}(B)$ there holds
\[
\|f\|_{L^{\infty}}\leq Pv(x),\qq{a.e.} x\in\{u>0\}\cap B.
\]
Then
\begin{equation}\label{modcap}
\|\nabla_Au_+\|_{L^2}^2\leq P\int(u_+v)\dd{\mu}.
\end{equation}
\end{cor}
\begin{proof}
    By Proposition \ref{cacc} we have,
    \[
    \int_{\qty{u>0}} |\nabla_A u_+|^2 \dd{\mu} \leq \int_{\qty{u>0}} u_+\|f\|_{L^{\infty}} \dd{\mu}
    \]
    a simple substitution gives
    \begin{align*}
       \int_{\qty{u>0}} |\nabla_A u_+|^2 \dd{\mu} &\leq \int_{\qty{u>0}} u_+Pv \dd{\mu} \\
       \|\nabla_Au_+\|_{L^2}^2&\leq P\int(u_+v)\dd{\mu}.
    \end{align*}
\end{proof}
Next we will need the following Lemma
\begin{lemma}\label{lem:2-1}
Let $\varphi(t)$ be a Young function and let $\phi(t)$ be defined by $\phi(t)=\varphi(t^2)$. Then for all $u\in L^{\phi}(B)$, 
\[
\|u^2\|_{L^{\varphi}}\leq \|u\|_{L^{\phi}}^2\leq 4\|u^2\|_{L^{\varphi}}.
\]
\end{lemma}
\begin{proof}
For the first inequality using Definition \ref{def:LuxNorm} we need to show that
\[
\int_{B} \varphi\qty(\frac{u^2}{\|u\|_{L^{\phi}}^2})d\mu\leq 1,   
\]
and because we have defined $\phi(t)=\varphi(t^2)$, we may write
\[
\int_{B} \varphi\qty(\frac{u^2}{\|u\|_{L^{\phi}}^2})d\mu = \int_{B} \phi\qty(\frac{u}{\|u\|_{L^\phi}})d\mu \leq 1
\] 
so we have $\|u^2\|_{L^{\varphi}}\leq \|u\|_{L^{\phi}}^2$. To show the second inequality 
we need to show that
\[
\int \varphi\qty(\frac{4u^2}{\|u\|_{L^{\phi}}^2})d\mu\geq 1.    
\]
Because we have defined $\phi(t)=\varphi(t^2)$, we can write
\[
\int \varphi\qty(\frac{4u^2}{\|u\|^2_{L^{\phi}}})
= \int \phi\qty(\frac{2u}{\|u\|_{L^{\phi}}}).
\]
By definition (\ref{LuxNorm}), $\|u\|_{L^\phi}$ is the smallest number such that $\int \phi\qty(\frac{u}{\|u\|_{L^{\phi}}})\leq 1$. Thus
\[
\int \varphi\qty(\frac{4u^2}{\|w\|^2_{L^{\phi}}})=\int \varphi\qty(\frac{u^2}{\nicefrac{\|u\|^2_{L^{\phi}}}{4}})
= \int \phi\qty(\frac{u}{\nicefrac{\|u\|_{L^{\phi}}}{2}}) \geq 1,
\]
which concludes the proof.
\end{proof}
We are now ready to prove the $L^{\infty}$ estimate in Theorem \ref{thm:sufficiency}, and the argument follows closely the proof of \cite[Proposition 27]{Luda}.

   \begin{thm}\label{thm:boundedness}
Let $L=\nabla\cdot A\nabla $ with bounded measurable non-negative semidefinite matrix $A$, and $d$ a metric on $\R^n$. Suppose also that the following Orlicz-Sobolev inequality holds for all $v\in \left(W^{1,2}_{A}\right)_{0}(B)$ and the metric ball $B\subset \R^n$ 
\begin{equation}\label{eq:OS}
\| v \|_{L^{\phi}(B)} \leq C(B)\|\nabla_A v\|_{L^2(B)},
\end{equation}
where $\phi$ is defined by $\phi(t)=\Phi(t^2)$ with $\Phi=\Phi_N$ from Definition  \ref{bumpfam}, for some $N>1$. Then the unique weak solution $u$ to (\ref{problem'}) satisfies
\[
\sup_B|u| \leq C \|f\|_{L^{\infty}(B)}.
\]
\end{thm} 
\begin{proof}
We first define the family of truncations
\[
u_k=(u-C_k)_+
\]
where
\[
C_k=\tau\|f\|_{L^{\infty}}\qty(1-c(k+1)^{-\frac{\epsilon}{2}}),\qq{} \tau\geq 1,
\]
and denote
\[
U_k\equiv\int_{B}\abs{u_k}^2 \dd{\mu}
\]
where $\dd{\mu}=\frac{\dd{x}}{\abs{B}}$.
Since $u_k\in \left(W^{1,2}_{A}\right)_{0}(B)$ for all $k$, using Hölder's inequality for Orlicz spaces (\ref{eq:Holder}) we can write
\begin{equation}\label{boundgoal}
\int u_{k+1}^2\dd{\mu}\leq C\|u_{k+1}^2\|_{L^{\Phi}}\cdot\|1\|_{L_{\qty{u_{k+1}>0}}^{\Tilde{\Phi}}}
\end{equation}
where the norms are taken with respect to the measure $\mu$.
Our first goal is to bound the first factor on the right. Note that if $u_{k+1}>0$ we have
\[
u>C_{k+1}=\tau\|f\|_{L^{\infty}}\qty(1-c(k+2)^{-\frac{\epsilon}{2}}),
\]
which implies that
\begin{align*}
u_k=\qty(u-C_k)_+&>c\tau\|f\|_{L^{\infty}}\qty[(k+1)^{-\frac{\epsilon}{2}}-(k+2)^{-\frac{\epsilon}{2}}] \\
&=c\tau\|f\|_{L^{\infty}}(k+1)^{-\frac{\epsilon}{2}}\qty[1-\qty(\frac{k+1}{k+2})^{\frac{\epsilon}{2}}]\\
&\geq c\tau\|f\|_{L^{\infty}}(k+1)^{-\frac{\epsilon}{2}}\qty(1-\frac{k+1}{k+2})\frac{\epsilon}{2}\qty(\frac{k+1}{k+2})^{\frac{\epsilon}{2}-1}\\
\end{align*}
Note that $\frac{k+1}{k+2}<1$ which allows us to conclude that 
\[
u_k\geq\frac{\epsilon}{2}c\tau\|f\|_{L^{\infty}}(k+2)^{-1-\frac{\epsilon}{2}}
\]
on the set where $u_{k+1}>0$, thus
\begin{equation}\label{modcapbound}
\|f\|_{L^{\infty}}\leq \frac{2}{c\tau\epsilon}(k+2)^{1+\frac{\epsilon}{2}}u_k\leq\frac{2}{c\epsilon}(k+2)^{1+\frac{\epsilon}{2}}u_k
\end{equation}
since $\tau\geq 1$.
Next, since $u$ is a weak solution it follows that $u-C_{k+1}$ is also a weak solution so (\ref{modcapbound}) implies we can use (\ref{modcap}) with $v=u_k$ and $P=\frac{2}{c\epsilon}(k+2)^{1+\frac{\epsilon}{2}}$ so we have
\begin{align*}
\int|\nabla_A u_{k+1}|^2\dd{\mu}
 &=\int\abs{\nabla_A\qty(u-C_{k+1})_+}^2 \dd{\mu} \\
    &\leq C\frac{2}{c\epsilon}(k+2)^{1+\frac{\epsilon}{2}}\int(u_{k+1}u_k)\dd{\mu} \\
    &\leq C\frac{2}{c\epsilon}(k+2)^{1+\frac{\epsilon}{2}}\int u_k^2\dd{\mu}.
\end{align*}
Applying (\ref{eq:OS}) and using Lemma \ref{lem:2-1} with $\varphi=\Phi$ we have
\[
\|u_{k+1}^2\|_{L^{\Phi}}\leq \|u_{k+1}\|_{L^{\phi}}^2\leq C\|\nabla_A u_{k+1}\|^2,
\]
which combining with the above inequality gives
\begin{equation}\label{Bound1}
     \|u_{k+1}^2\|_{L^{\Phi}}\leq C(k+2)^{\frac{2+\epsilon}{2}}\int u_k^2 \dd{\mu}.
\end{equation}
Now we want to bound the second factor on the right hand side of (\ref{boundgoal}), i.e. $\|1\|_{L^{\Tilde{\Phi}}}$. Consider the function
\[
\Gamma(t):= \frac{1}{\Tilde{\Phi}^{-1}(\frac{1}{t})}
\]
and note that
\[
\int_{\{u_{k+1}>0\}}\Tilde{\Phi}\qty(\frac{1}{a})\dd{\mu}=\Tilde{\Phi}\qty(\frac{1}{a})\mu(\{u_{k+1}>0\})
\]
for all $a>0$. Now let
\[
a=\Gamma\qty(\mu(\{u_{k+1}>0\}))=\frac{1}{\Tilde{\Phi}^{-1}\qty(\frac{1}{\mu(\{u_{k+1}>0\})})}
\] 
so that
\[
\int_{\{u_{k+1}>0\}}\Tilde{\Phi}\qty(\frac{1}{a})\dd{\mu}=1,
\]
and therefore
\begin{equation}\label{1bound}
\|1\|_{L^{\Tilde{\Phi}}\qty(\{u_{k+1}>0\})}\leq a = \Gamma\qty(\mu(\{u_{k+1}>0\})).
\end{equation}
Now recall that we showed 
\[
\{u_{k+1}>0\}\subset\qty{u_k>\frac{\epsilon}{2}c\tau\|f\|_{L^{\infty}}(k+2)^{-1-\frac{\epsilon}{2}}}
\]
where $\tau\geq 1$ which follows from the observation that 
\[
u_{k+1}>0
\]
implies
\[
u_k>\tau\|f\|_{L^{\infty}}\qty(1-c(k+2)^{-\frac{\epsilon}{2}}).
\]
Using Chebyshev's inequality this gives
\begin{equation}\label{chev}
\mu\left(\{u_{k+1}>0\}\right)\leq\mu\left(\qty{u_k>\frac{\epsilon}{2}c\tau\|f\|_{L^{\infty}}(k+2)^{-1-\frac{\epsilon}{2}}}\right)\leq\frac{4}{c^2\tau^2\|f\|_{L^{\infty}}^2\epsilon^2}(k+2)^{2+\epsilon}\int u_k^2\dd{\mu}.
\end{equation}
Combining (\ref{1bound}) and (\ref{chev}) we obtain
\begin{equation}\label{Bound2}
    \|1\|_{L_{\qty{u_{k+1}>0}}^{\Tilde{\Phi}}}\leq \Gamma\qty(C(k+2)^{2+\epsilon}\int u_k^2).
\end{equation}
Finally substituting (\ref{Bound1}) and (\ref{Bound2}) into (\ref{boundgoal}) we conclude that
\begin{align*}
\int u_{k+1}^2\dd{\mu}&\leq C(k+2)^{\frac{2+\epsilon}{2}}\int u_k^2\cdot\Gamma\qty(C(k+2)^{2+\epsilon}\int u_k^2) \\
U_{k+1} &\leq C(k+2)^{\frac{2+\epsilon}{2}}U_k\Gamma\qty(C(k+2)^{2+\epsilon}U_k).
\end{align*}
This estimate is the same as the one obtained in the proof of \cite[Theorem 30]{Luda}, so the rest of the proof can be repeated verbatim which gives 
\[
\sup_B|u| \leq C\|f\|_{L^{\infty}(B)}.
\]
\end{proof}

\section{Almost Necessity}\label{NecExtUnq}

In this section we demonstrate the almost necessity of an Orlicz-Sobolev inequality for the existence, uniqueness, and boundedness of solutions to (\ref{Problem}), namely we prove Theorem \ref{thm:necessity}.
We start with two simple technical lemmas.
\begin{lemma}\label{supBdd}
Let $\varphi:\R\to[0,\infty]$ be a Young function and let $\Tilde{\varphi}$ be the convex conjugate, or dual, of $\varphi$ as defined in Definition \ref{dual}. Let $B\subset \R^n$ be any ball, and define  
\[
X=\{f\in L^{\tilde{\varphi}}\mid \int_B\Tilde{\varphi}(|f|)d\mu \leq 1\} \qquad \text{and} \qquad Y=\{f\in L^{\tilde{\varphi}}\mid f\geq 0 \text{ and } \int_B\Tilde{\varphi}(f)d\mu\leq 1\}.
\]
Then, 
\begin{equation}
    \sup_X \int_B w^2f\ d\mu = \sup_Y \int_B w^2f\ d\mu
\end{equation}
for any $w\in Lip_0(B)$.
\end{lemma}
\begin{proof}
    First note that $Y\subset X$. Thus, it suffices to show that 
    \[
    \int_B w^2 g \leq \sup_Y \int_B w^2f \qquad \text{for all $g\in X\backslash Y$}.
    \]
    Let $g\in X\backslash Y$, and write $g^+$ and $g^-$ for the positive and negative parts of $g$ respectively. 
    
    Note that since $\tilde{\varphi}$ is even and non-decreasing on $[0,\infty)$ and non-negative on $\R$, we have 
 \begin{align*}
    \int_B \tilde{\varphi}(g^+)
    & \leq \int_B \tilde{\varphi}(g^++g^-)\\
    & =\int_B \tilde{\varphi}(|g|)\\
    & \leq 1\quad  \text{for all $g\in X$}.
    \end{align*}
In particular, we conclude that $g^+\in Y$. Therefore,
    \begin{align*}
        \int_B w^2 g 
         = \int_B w^2 g^+ - \int_B w^2 g^-
         \leq \sup_{f\in Y} \int_B w^2 f - \int w^2g^-
         \leq \sup_{f\in Y} \int_B w^2 f.
    \end{align*}
    The last inequality follows from the fact that since $w^2$ and $g^-$ are both non-negative, it must be the case that $\int_B w^2g^-$ is also non-negative.
\end{proof}
\begin{lemma}\label{boundLuxNorm}
Let $\varphi:\R\to [0,\infty]$ be a Young function and $\Tilde{\varphi}$ be its dual. Furthermore, define the sets $X$ and $Y$ as in Lemma \ref{supBdd}. Then, for all $w\in Lip_0(B)$,
\[
\|u\|_{L^{\varphi}} \leq \sup_{f\in Y}\int_B uf\ d\mu.
\]
\end{lemma}
\begin{proof}
As before let $d\mu=dx/|B|$, and recall from (\ref{eq:Or_eqiuv}) that 
\[
\|u\|_{L^{\varphi}} \leq |u|_{L^{\varphi}} =\sup\left\{ \int ug\ d\mu : \int \Tilde{\varphi}(g)\leq 1 \right\}
=\sup\left\{ \int ug\ d\mu : \int \Tilde{\varphi}(|g|)\leq 1 \right\},
\]
where the last equality is due to the fact that $\Tilde{\varphi}$ is even by definition of a Young function.
Finally, using Lemma \ref{boundLuxNorm} we have
\[
\sup\left\{ \int ug\ d\mu : \int \Tilde{\varphi}(|g|)\leq 1 \right\}=\sup_{g\in X}\int_B ug\ d\mu \leq \sup_{g\in Y}\int_{B} ug\ d\mu,
\]
which concludes the result.
\end{proof}
We are now ready to prove Theorem \ref{thm:necessity}, which we state again here for convenience.
\begin{thm}\label{necessity}
    Let $\varphi:[0,\infty] \rightarrow [0,\infty]$ be a Young function that satisfies $\varphi(t) > t$ for all $t>0$, and let $\phi(t)=\varphi(t^2)$. Additionally, let $f\in L^{\Tilde{\varphi}}(B)$ and assume that all
    weak solutions, $u\in \left(W^{1,2}_A\right)_0(B)$,
    to (\ref{Problem}) satisfy the global boundedness estimate $\sup_B|u|\leq
    C\|f\|_{L^{\Tilde{\varphi}}(B)}$. Then the following Orlicz-Sobolev inequality holds:
    \begin{align*}
        \|v\|_{L^\phi(B,d\mu)} \leq C\|\nabla_A v\|_{L^2(B,d\mu)} \quad \text{ for all } v\in \left(W_A^{1,2}\right)_0(B).
    \end{align*}
\end{thm}

Note that the global boundedness condition is different from that in the sufficiency result. We previously demonstrated that a $\phi-2$ Orlicz-Sobolev inequality with sufficiently large $\phi$ gives the estimate, $\sup_B|u|\leq C\|f\|_{L^\infty(B)}$, for all weak solutions, $u$, to (\ref{Problem}) with $f\in L^{\infty}(B)$. However, in order to prove necessity of a $\phi-2$ Orlicz-Sobolev inequality we require a stronger condition; namely, that all weak solutions to (\ref{Problem}) with the right hand side in a larger class, i.e. $f\in L^{\Tilde{\varphi}}(B)$, are bounded.
Hence the term ``almost necessity". 

\begin{proof}
    The proof is similar to the proof of Lemma 102 in \cite{Sawyer}, and the proof in Sections 1 and 2 of Chapter 9 in \cite{Luda2}.
    
        Let $u\in \left(W^{1,2}_A\right)_0$ be a weak solution to (\ref{Problem}), i.e. 
    \[
    \int_B \nabla \psi\cdot A \nabla u \ d\mu=-\int_B \psi f\ d\mu
    \]
    for all $\psi\in Lip_0(B)$, and assume $f\geq 0$. For any $w\in Lip_0(B)$, we therefore have,
    \[
    -\int_B \nabla w^2 \cdot A \nabla u=\int_B w^2f,
    \]
    since $w^2\in Lip_0(B)$. By applying the chain rule to $\nabla w^2$ and using the inner product from Definition \ref{innerProduct}, we see that 
    \begin{equation}\label{initialInequality}
    \int_B w^2f = -2\int_B w \langle\nabla w, \nabla u\rangle \leq 2\qty(\int_B w^2 [\nabla u]^2_A)^{\frac{1}{2}}\qty(\int_B [\nabla w]^2_A)^{\frac{1}{2}},
    \end{equation}
     where the inequality follows from an application of the Cauchy-Schwartz inequality followed by the H\"older inequality.
    Now analyzing the first term in the above inequality, we observe that,
    \begin{equation}\label{term1intermediary}
    \int_B w^2[\nabla u]^2_A = \int_B w^2 \nabla u \cdot A\nabla u.
    \end{equation}
    Furthermore, since $Lip_0(B)$ is dense in $\left(W^{1,2}_A\right)_0$ we can take $w^2u$ as a test function in the definition of a weak solution to obtain
    \begin{align*}
        -\int w^2uf 
        & = \int \nabla(w^2u)\cdot A\nabla u\\
        & = \int (2w\nabla w u + w^2\nabla u)\cdot A \nabla u\\
        & = 2\int wu \nabla w\cdot A\nabla u + \int_B w^2\nabla u\cdot A\nabla u.
    \end{align*}
    Therefore,
    \[
    \int w^2\nabla u\cdot A\nabla u = -2\int wu\nabla w\cdot A \nabla u - \int w^2uf.
    \]
    
    Hence, (\ref{term1intermediary}) becomes
    \begin{align*}
        \int w^2[\nabla u]^2_A
        & = \int w^2 \nabla u\cdot A\nabla u\\
        & = -2\int wu\nabla w\cdot A \nabla u - \int w^2uf\\
        & = -2\int \langle u\nabla w,w\nabla u\rangle - \int uw^2 f\\
        &\leq \frac{1}{2} \int w^2[\nabla u ]^2_A+8\int u^2[\nabla w]^2_A+\int |u|w^2 |f|,
    \end{align*}
    where the final estimate follows from the Cauchy-Schwartz Inequality followed by Young's Inequality.
    Absorbing the first term on the right to the left-hand side, results in 
    \begin{align*}
    \int w^2[\nabla u]^2_A &\leq C\qty(\sup_B |u|)^2\int[\nabla w]^2_A+C\qty(\sup_B |u|)\int w^2 |f|\\
    &\leq C\max\qty{{\qty(\sup_B |u|)^2\int[\nabla w]^2_A ,\qty(\sup_B |u|)\int w^2 |f|}}\\
    & = C\max\qty{{\qty(\sup_B |u|)^2\int[\nabla w]^2_A ,\qty(\sup_B |u|)\int w^2 f }},
    \end{align*}
    where the last equality follows from the assumption that $f$ is non-negative.
    We claim that comparing $\int w^2[\nabla u]_A^2$ to either term inside the maximum results in equivalent inequalities. First observe that comparing $\int w^2[\nabla u]_A^2$ to $\left ( \sup_B|u|\right )^2\int [\nabla w]_A^2$ and combining with (\ref{initialInequality}) results in
    \begin{align*}
        \int w^2 f &\leq C\qty(\sup_B |u|)\int[\nabla w]^2_A\\
        &=C \qty(\sup_B |u|)\|\nabla_A w\|^2_{L^2} \\
        &\leq C\|f\|_{L^{\Tilde{\varphi}}}\|\nabla_A w\|^2_{L^2}
    \end{align*}
    where the final inequality follows from the global boundedness estimate.
    Now comparing $\int w^2[\nabla u]_A^2$ to $\sup_B|u|\int w^2 f$ and combining with (\ref{initialInequality}) results in
    \[
        \int w^2f
        \leq \left (\sup_B|u|\int w^2 f\right )^{1/2} \left(\int[\nabla w]_A^2\right)^{1/2}.
    \]
   Combining with the global boundedness estimate, this becomes
    \begin{equation}\label{eq:1}
        \int w^2f
         \leq \|f\|_{L^{\Tilde{\varphi}}} \int[\nabla w]_A^2,
    \end{equation}
    which is the same estimate as above.
Using the equivalent definition of the Orlicz norm (\ref{OrliczNorm}) and Lemma \ref{supBdd}, we know that 
    \[
        |w^2|_{L^{\varphi}}
        =\sup\left\{\int_B w^2 f:\ \int_B \Tilde{\varphi}(f)\leq 1\  \text{and}\  f\geq 0\right\}
        =\sup\left\{\int_B w^2 f:\ \|f\|_{L^{\Tilde{\varphi}}}\leq 1\  \text{and}\  f\geq 0\right\}.
    \]
    Combining with (\ref{eq:1}) and (\ref{eq:Or_eqiuv}) gives
 \[
    \|w^2\|_{L^\varphi}  \leq C\|\nabla_A w\|^2_{L^2}.
    \]
    To obtain the desired $\phi-2$ Orlicz-Sobolev Inequality it remains to show that $\|w\|_{L^{\phi}}^2\leq C\|w^2\|_{L^{\varphi}}$. This follows immediately from the second inequality in Lemma \ref{lem:2-1}.
    Hence,
\[
\|w\|_{L^\phi}^2\leq 4\|w^2\|_{L^\varphi} \leq C\|\nabla_A w\|_{L^2}^2.
\]
By the density of $Lip_0(B)$ in $W^{1,2}_0(B)$ we obtain the desired Orlicz-Sobolev inequality,
\[
\|v\|_{L^{\phi}(B,d\mu)}\leq \|\nabla_A v\|_{L^2(B,d\mu)} \text{ for all } v\in \left(W_A^{1,2}\right)_0
\]
\end{proof}

\section{Sharpness}\label{Counterexamples}

In this section, we demonstrate a weak degree of sharpness of our results. More precisely, we show that even though the requirement on the right hand side function $f$ in the sufficiency result, Theorem \ref{thm:sufficiency}, is stronger then the one in Theorem \ref{thm:necessity}, it cannot be significantly relaxed. Namely, there exist an operator $A$ and a function $u\in \left(W^{1,2}_{A}\right)_{0}$ such that (1) a $\Psi-2$ Orlicz-Sobolev inequality holds in a subunit metric ball $B$ with $\Psi(t)= t^{2}(\ln t)^{N}$, $N>1$, for all $t>1$; (2) $Lu\in L^{\tilde{\Phi}_M}$ with $\Phi_M(t)= t(\ln t)^{M}$, $M>2+2N$, for all $t>1$; (3) $u$ is unbouned at the origin. To set the stage, we first provide similar constructions in the case of the Laplacian operator, and a finitely degenerate operator.

\subsection{Laplacian counterexample}
Recall that in general we are concerned with the following divergence form operator
$Lu=\nabla\cdot A\nabla u$. Now consider the two dimensional case of $\R^{2}$ and
let
\[
A=\begin{pmatrix}
1&0 \\
0&1
\end{pmatrix},
\]
so $L=\Delta$, the Laplace operator.
For a generic $u$ changing to polar coordinates gives  
\[
\Delta u =\frac{1}{r}\frac{\partial}{\partial r}\qty(r\frac{\partial u}{\partial r})+\frac{1}{r^2}\frac{\partial^2 u}{\partial \theta^2}.
\]
Now we choose a weak solution $u$ that is unbounded at the origin. 
The power $\alpha$ helps us control the integrability of this 
unbounded function. Define
\[
u=\qty(\ln\frac{1}{r})^{\alpha},
\]
where $0<\alpha< 1/2$, so one can check that $u\in W^{1,2}(B(0,1/2))$. Since this function does not depend on $\theta$, we have
\[
Lu=\Delta u = \frac{1}{r^2}\alpha(\alpha-1)\ln\qty(\frac{1}{r})^{\alpha-2}.
\]
Thus with 
\[
f := \frac{1}{r^2}\alpha(\alpha-1)\qty(\ln\frac{1}{r})^{\alpha-2}
\]
$u$ is a weak solution to $Lu=f$ which is unbounded at the origin.
We now calculate the
$L^q$ norm of $f$ in the ball $B=B(0,1/2)$
\[
||f||_{L^{q}(B)}=\int_{0}^{2\pi}\int_0^{\frac{1}{2}}\abs{f(r)}^qr\dd{r}\dd\theta=2\pi\alpha(\alpha-1)\int_{0}^{\frac{1}{2}}\abs{\frac{1}{r^2}\qty(\ln\frac{1}{r})^{\alpha-2}}^q\ r\dd{r}
\approx \int_{0}^{1/2}\qty(\ln\frac{1}{r})^{q(\alpha-2)}\frac{dr}{r^{2q-1}}.
\]
The integral on the right is finite provided $q<1$ or $q= 1$ and $\alpha <1$. In particular, $u=\ln(1/r)^{1/4}$ is an unbounded weak solution to $\Delta u=f$ with $f\in L^{q}(B)$, $q=1=n/2$.  

On the other hand, if $u$ is a weak solution to $Lu=f$ and $f\in L^{q}(B)$ with $q>n/2$, Theorem 8.16 in \cite{GilTrud} gives that 
\begin{equation*}
    \sup_{B}{|u|} \leq   \sup_{\partial B}|u|+ C ||f||_{q}<\infty.
\end{equation*}

Furthermore, for the Laplace operator, the associated subunit metric space coincides with the Euclidean $\R^n$, and we have the following Sobolev Inequality 

\begin{equation}\label{Sob-classic}
\left(\frac{1}{|B|}\int_{B}|w|^{2\sigma}\right)^{\frac{1}{2\sigma}}\leq Cr\left(\frac{1}{|B|}\int_{B}|\nabla_A w|^{2}\right)^{\frac{1}{2}}+C\left(\frac{1}{|B|}\int_{B}|w|^{2}\right)^{\frac{1}{2}}
\end{equation}
for $\sigma \leq \frac{n}{n-2}$ and so $\sigma'=n/2$ is the dual of $\sigma$.

\subsection{Degenerate counterexamples}
The following two examples are based on the examples constructed in \cite{Luda2} (see Section 3 Chapter 9). 
Let $L=\nabla\cdot A\nabla$ with 
\[A=\begin{pmatrix}
1 & 0 \\
0 & g(x)^2 \\
\end{pmatrix},\]
where $g(0)=0$, $g$ is positive away from the origin, and $g=\psi'$ where $\psi$ is smooth, even, strictly convex on $\R$ and $\psi(0)=0$. Moreover, we will assume that $g(x)/x\to 0$ and $\psi(x)/x^{2}\to 0$ as $x\to 0$. 
Since the operator $L$ is elliptic away from the $y$-axis, and translation invariant with respect to the $y$ variable, we may restrict our attention to the ball $B=B(0,\rho)$ centered at the origin, with radius $\rho$ sufficiently small.
Define the function $u$ by
\begin{equation}\label{eq:u_def}
u(x,y):=\chi\qty(\frac{y}{\psi(x)})\ln{\frac{1}{x}},
\end{equation}
where $\chi(s)$ is a smooth odd function on $\R$ such that $\chi(s)=1$ for $s\in[-1,1]$ and $\chi(s)=0$ for $s\in\R\smallsetminus[-2,2]$. First note that the function $u$ is supported in the narrow region along the $x$-axis, where $|y|\leq 2\psi(x)$.
Next we calculate
\begin{align*}
 u_y &=\chi'\qty(\frac{y}{\psi(x)})\frac{1}{\psi(x)}\ln{\frac{1}{x}},\\
 u_{yy}&=\chi''\qty(\frac{y}{\psi(x)})\frac{1}{\psi(x)^2}\ln{\frac{1}{x}},\\ 
 u_x &=\chi'\qty(\frac{y}{\psi(x)})\qty(\frac{-y\psi'(x)}{\psi(x)^2})\ln{\frac{1}{x}}-\frac{1}{x}\chi\qty(\frac{y}{\psi(x)}), \\
 u_{xx}&=\chi''\qty(\frac{y}{\psi(x)})\qty(\frac{y\psi'(x)}{\psi(x)^2})^2\ln{\frac{1}{x}}+\chi'\qty(\frac{y}{\psi(x)})\qty(\frac{2y\psi'(x)^2}{\psi(x)^3}-\frac{y\psi''(x)}{\psi(x)^2})\ln{\frac{1}{x}}\\
&\quad+\frac{2}{x}\chi'\qty(\frac{y}{\psi(x)})\qty(\frac{y\psi'(x)}{\psi(x)^2})+\frac{1}{x^2}\chi\qty(\frac{y}{\psi(x)}).
\end{align*}
To make further estimates, we will write $a\approx b$ for any two given functions $a$ and $b$ to imply that the two inequalities
\[
C_1 a \leq b \leq C_2 a
\]
hold for all elements of the domain of $a$ and $b$ and for some constants $C_1,C_2>0$. Define 
\[
f(x,y):=Lu=u_{xx}+g(x)^2u_{yy},
\]
using the properties of $g$, $\psi$, and $\chi$, we then have
\begin{equation}\label{eq:f_est}
 |f(x,y)|\approx \frac{1}{x^2}+\ln \frac{1}{x}\qty(\frac{|\psi''(x)|}{\psi(x)})+\ln \frac{1}{x}\qty(\frac{\psi'(x)}{\psi(x)})^2+\frac{1}{x}\qty(\frac{\psi'(x)}{\psi(x)}),   \end{equation}
and $f$ is supported in $|y|\leq 2\psi(x)$.

\subsubsection*{Finite vanishing}
Fix $m\geq 1$ and let
\[
\psi(x)=\frac{1}{m+1}x^{m+1}.
\]
Differentiating gives
\[
g(x)=\psi'(x)=x^{m},\  \  \psi''(x)= mx^{m-1},
\]
which combined with (\ref{eq:f_est}) implies
\[
|f(x,y)|\approx\frac{1}{x^{2}}\ln\frac{1}{x}.
\]
Recall that $f$ is supported where $|y|\leq 2\psi(x)$, so we can estimate the $L^q$ norm in the ball $B$
\[
\int_{B}|f(x,y)|^{q}dxdy\lesssim \int_{0}^{\rho}\frac{1}{x^{2q}}\left(\ln\frac{1}{x}\right)^{q}\psi(x)dx
\approx \int_{0}^{\rho}\frac{1}{x^{2q-m-1}}\left(\ln\frac{1}{x}\right)^{q}dx.
\]
The right hand side is finite if and only if $q<\frac{m+2}{2}$. We now verify that the function $u$ belongs to the Sobolev space $W^{1,2}_{A}(B)$, i.e. $u\in L^{2}(B)$ and
\[
\int_{B}|\nabla_{A}u|^{2}dxdy=\int_{B}\left(|u_x|^{2}+g(x)^{2}|u_y|^{2}\right)dxdy<\infty.
\]
Using the expressions for $u_x$ and $u_y$ and the estimates for $\psi'$ and $\psi''$ we have for $|y|\leq 2\psi(x)$
\[
|u_x|^{2}+g(x)^{2}|u_y|^{2}\approx \frac{1}{x^2}\left(\ln\frac{1}{x}\right)^{2}+\frac{g(x)^2}{\psi(x)^2}\left(\ln\frac{1}{x}\right)^{2}\approx\frac{1}{x^2}\left(\ln\frac{1}{x}\right)^{2},
\]
where for the last equality we used $g=\psi'$. Altogether we obtain
\[
\int_{B}|\nabla_{A}u|^{2}dxdy\approx \int_{0}^{\rho}\frac{1}{x^2}\left(\ln\frac{1}{x}\right)^{2}\psi(x)dx
\approx \int_{0}^{\rho}\frac{1}{x^{1-m}}\left(\ln\frac{1}{x}\right)^{2}dx,
\]
which is finite for all $m>0$. It is easy to see that $u\in L^{2}(B)$, so that $u\in W^{1,2}_{A}(B)$. Moreover, we have $u(x,\psi(x))=\ln(1/x)$, so $u$ is unbounded at the origin. Thus, for any $q<\frac{m+2}{2}$ we obtain that $u$ is an unbounded weak solution to $Lu=f$ with  $f\in L^{q}(B)$.

On the other hand Proposition 74 in \cite{Sawyer} implies
the following  Sobolev inequality
\begin{equation}\label{eq:sob_saw}
    \left\{ \frac{1}{|B|}\int_{B}{|w|^{2\sigma}} \right\}^{\frac{1}{2\sigma}} 
    \leq 
    Cr 
    \left\{ \frac{1}{|B|}\int_{B}{|\nabla_{A} w|^{2}} \right\}^{\frac{1}{2}}
    + 
    C 
    \left\{
    \frac{1}{|B|} \int_{B}{|w|^2}
    \right\}^{\frac{1}{2}}
\end{equation}
for all $w \in W^{1,2}_{0}(B)$, where $\sigma'=(m+2)/2$. Theorem 8 in \cite{Sawyer} then implies that
if $u$ is a weak solution to $Lu=f$ and $f\in L^{q}(B)$ with $q>\frac{m+2}{2}$, it is locally bounded.

\subsubsection*{Infinite vanishing} 
We now consider the case when the function $g$, and therefore $\psi$, vanishes to infinite order at the origin. Namely, fix $\alpha>0$ and define
\[
\psi(x):=x^{\alpha+1}e^{-\frac{1}{x^{\alpha}}},
\]
so that 
\[
g(x)=\psi'(x)=\alpha e^{-\frac{1}{x^{\alpha}}}+(\alpha+1)x^{\alpha}e^{-\frac{1}{x^{\alpha}}}\approx e^{-\frac{1}{x^{\sigma}}},
\]
and
\[
\qty(\frac{\psi'(x)}{\psi(x)})^2\approx\frac{\psi''(x)}{\psi(x)}\approx \frac{1}{x^{2\alpha+2}}.
\]
Combining with (\ref{eq:f_est}) gives
\[
|f(x,y)|\approx\frac{1}{x^{2\alpha+2}}\ln\frac{1}{x},
\]
and $f$ is supported in $|y|\leq 2\psi(x)$. Note that $f$ does not belong to $L^{\infty}(B)$ since it is unbounded at the origin, so we look for an appropriate Orlicz space for the function $f$. Recall the family of Young functions $\Phi_N$ given in Definition \ref{bumpfam}. The following Orlicz-Sobolev inequality has been shown in \cite{Luda}
\[ 
\| w \| _{L^{\Phi}(B)}  \leq C \| \nabla_{A}w \|_{L^1(B)} \quad\text{for}\  \    w \in \qty(W^{1,1}_{A})_{0}(B),
\] 
if $\Phi=\Phi_N$, $N\geq 1$, and $\alpha N<1$. Here, $B$ is a sufficiently small subunit metric ball centered at the origin. Letting $w=v^{2}$ and using Cauchy-Schwartz inequality and Lemma \ref{lem:2-1} we then obtain
\begin{equation}\label{eq:sob_inf}
\| v \| _{L^{\Psi}(B)}  \leq C \| \nabla_{A}v \|_{L^2(B)}+C\|v\|_{L^2(B)} \quad\text{for}\  \    v \in \qty(W^{1,2}_{A})_{0}(B),
\end{equation}
with $\Psi$ defined by $\Psi(t)=\Phi(t^{2})$, i.e. $\Psi(t)\approx t^2(\ln t)^N$ for all $t> 1$, provided $N\alpha<1$. To make analogy to the finite type case note that (\ref{eq:sob_saw}) is (\ref{eq:sob_inf}) with
\[
\Psi(t)=t^{2\sigma},\quad\text{or}\  \  \Phi(t)=t^\sigma.
\]
Thus, just as in the finite type case (or elliptic case), we expect all weak solutions to $Lu=f$ to be bounded provided $f$ belongs to a slightly smaller space than the dual of $L^{\Phi}$, which is $L^{\tilde{\Phi}}$. On the other hand, if $f$ is in a slightly bigger space than $L^{\tilde{\Phi}}$ we expect there to exist an unbounded weak solution to $Lu=f$. We have already shown in Theorem \ref{thm:boundedness} that every weak solution $u\in \left(W^{1,2}_{A}\right)_{0}(B)$ to $Lu=f$ with $f\in L^{\infty}(B)\subsetneq L^{\tilde{\Phi}}(B)$ is bounded.\\
Finally, let $u$ be defined by (\ref{eq:u_def}), and one can verify that $u\in W^{1,2}_{A}(B)$. Recall that with $f=Lu$ we have (\ref{eq:f_est}), i.e.
\[
|f(x,y)|\approx\frac{1}{x^{2\alpha+2}}\ln\frac{1}{x} 
\]
supported in $|y|\leq 2\psi(x)$. We now would like to find a Young function $\theta$ so that $f\in L^{\theta}(B)$ and we expect $L^{\theta}(B)$ to be larger than $L^{\tilde{\Phi}}(B)$ (analogous to $q<\sigma'$). Let $\theta=\tilde{\Phi}_M$, $M\geq 1$, using the estimates from \cite{Luda} we have
\[
\theta(s)\leq Ms^{1-\frac{1}{M}}e^{s^{\frac{1}{M}}} 
\]
for $s\geq(2M)^{M}$.
Therefore,
\begin{align*}
\int_{B}\theta(f(x,y))dxdy&\approx\int_{0}^{\rho}\theta\left(\frac{1}{x^{2\alpha+2}}\ln\frac{1}{x} \right)\psi(x)dx\\
&\approx \int_{0}^{\rho}\left(\ln\frac{1}{x}\right)^{1-\frac{1}{M}}\frac{1}{x^{(2\alpha+2)(1-1/M)}}\exp{\left(\ln\frac{1}{x}\right)^{\frac{1}{M}}\cdot \frac{1}{x^{(2\alpha+2)/M}}-\frac{1}{x^{\alpha}}}dx.
\end{align*}
In order for this integral to be finite we must require
\[
  \frac{2\alpha+2}{M}<\alpha, 
\]
which implies
\[
M > 2 + \frac{2}{\alpha} > 2 + 2N,
\]
since $\alpha N<1$. Thus $f\in L^{\tilde{\Phi}_{M}}(B)$ for $M>2+2N$, and since $M>N$ (i.e. $L^{\Phi_M}\subsetneq L^{\Phi_N}$) we have
\[
L^{\tilde{\Phi}_N}\subsetneq L^{\tilde{\Phi}_M}
\]
as expected. Therefore, there exists an unbounded weak solution to $Lu=f$ with $f\in L^{\tilde{\Phi}_{M}}(B)$.


\section*{Acknowledgements}

The first three authors would like to first and foremost thank Luda Korobenko. Without her incredible kindness, patience, and generosity this research would never have taken place, and our summers would have been much more dull.

\end{document}